\documentclass[12pt]{article}

\usepackage{amsmath}
\usepackage{amssymb}
\usepackage{latexsym}
\usepackage{color}

\newtheorem{assumption}{Assumption}
\def\qed{ \ \vrule width.2cm height.2cm depth0cm\smallskip}
\newenvironment{proof}{\noindent {\bf Proof.\/}}{$\qed$\vskip 0.1in}

\newcommand{\ol}{\overline}
\newcommand{\ul}{\underline}

\newcommand{\eps}{\varepsilon}

\newcommand{\ba}{\begin{array}}
\newcommand{\ea}{\end{array}}
\newcommand{\be}{\begin{equation}}
\newcommand{\ee}{\end{equation}}
\newcommand{\bea}{\begin{eqnarray}}
\newcommand{\eea}{\end{eqnarray}}
\newcommand{\beaa}{\begin{eqnarray*}}
\newcommand{\eeaa}{\end{eqnarray*}}

\def\dbE{\mathbb{E}}
\def\dbF{\mathbb{F}}

\def\dbH{\mathbb{H}}

\def\dbL{\mathbb{L}}

\def\dbN{\mathbb{N}}
\def\dbP{\mathbb{P}}
\def\dbR{\mathbb{R}}

\def\a{\alpha}

\def\d{\delta}
\def\e{\varepsilon}

\def\l{\lambda}
\def\m{\mu}

\def\si{\sigma}

\def\f{\varphi}
\def\th{\theta}
\def\o{\omega}

\def\Th{\Theta}

\def\O{\Omega}
\def\cA{{\cal A}}

\def\cF{{\cal F}}

\def\cL{{\cal L}}

\def\cO{{\cal O}}

\def\no{\noindent}

\def\q{\quad}

\def\pa{\partial}
\def\cd{\cdot}

\def\qed{ \hfill \vrule width.25cm height.25cm depth0cm\smallskip}

\newcommand{\basa}{\begin{assumption}}
\newcommand{\easa}{\end{assumption}}

\newcommand{\bas}{\begin{assum}}
\newcommand{\eas}{\end{assum}}

\def\limsup{\mathop{\overline{\rm lim}}}
\def\liminf{\mathop{\underline{\rm lim}}}

\def\pa{\partial}

 \def\cd{\cdot}

\def\1{{\bf 1}}

\def\Hbmo{\dbH_{\mbox{\sc\tiny bmo}}}

\def\:{\!:\!}
\def\reff#1{{\rm(\ref{#1})}}
\def \proof{{\noindent \bf Proof\quad}}

at 9pt

\begin{document}

\newtheorem{thm}{Theorem}[section]
\newtheorem{lem}[thm]{Lemma}
\newtheorem{cor}[thm]{Corollary}
\newtheorem{prop}[thm]{Proposition}
\newtheorem{rem}[thm]{Remark}
\newtheorem{eg}[thm]{Example}
\newtheorem{defn}[thm]{Definition}
\newtheorem{assum}[thm]{Assumption}

\renewcommand {\theequation}{\arabic{section}.\arabic{equation}}
\def\thesection{\arabic{section}}

\title{Large Deviations for Non-Markovian Diffusions
             \\
             and a Path-Dependent Eikonal Equation}
\author{Jin  {\sc Ma}\footnote{University of Southern California, Department of Mathematics, jinma@usc.edu. Research supported in part by NSF grant DMS 1106853.}   
~~
Zhenjie {\sc Ren}\footnote{CMAP, Ecole Polytechnique Paris, ren@cmap.polytechnique.fr. Research supported by grants from R\'egion Ile-de-France.}
~~ Nizar {\sc Touzi}\footnote{CMAP, Ecole Polytechnique Paris, nizar.touzi@polytechnique.edu. Research supported by the ERC 321111 Rofirm, the ANR Isotace, and the Chairs Financial Risks (Risk Foundation, sponsored by Soci\'et\'e G\'en\'erale) and Finance and Sustainable Development (IEF sponsored by EDF and CA). }
   ~~ Jianfeng {\sc Zhang}\footnote{University of Southern California, Department of Mathematics, jianfenz@usc.edu. }
   %Research supported in part by NSF grant \#DMS 1008873.}
\thanks{We are grateful for Mike Tehranchi who introduced us to the application of large deviations to the problem of implied volatility asymptotics.}
}
	\maketitle

\begin{abstract}
This paper provides a large deviation principle for Non-Markovian, Brownian motion driven stochastic differential equations with random coefficients. Similar to Gao \& Liu \cite{GL}, this extends the corresponding results collected in Freidlin \& Wentzell \cite{FreidlinWentzell}. However, we use a different line of argument,  adapting the PDE method of Fleming \cite{Fleming} and Evans \& Ishii \cite{EvansIshii} to the path-dependent case, by using backward stochastic differential techniques. Similar to the Markovian case, we obtain a characterization of the action function as the unique bounded solution of a path-dependent version of the Eikonal equation. Finally, we provide an application to the short maturity asymptotics of the implied volatility surface in financial mathematics.
\end{abstract}

\noindent{\bf Key words:} Large deviations, backward stochastic differential equations, viscosity solutions of path dependent PDEs.

\noindent{\bf AMS 2000 subject classifications:}  35D40, 35K10, 60H10, 60H30.

%\vfill\eject

\section{Introduction}
\label{sect-Introduction}
\setcounter{equation}{0}

The theory of large deviations is concerned with the rate of convergence of a vanishing sequence of probabilities $\big(\dbP[A_n]\big)_{n\ge 1}$, where $(A_n)_{n\ge 1}$ is a sequence of {\it rare} events. after convenient scaling and normalization, the limit is called {\it rate function}, and is typically represented in terms of a control problem. 

The pioneering work of Freidlin and Wentzell \cite{FreidlinWentzell} considers rare events induced by Markov diffusions. The techniques are based on the Girsanov theorem for equivalent change of measure, and classical convex duality. An important contribution by Fleming \cite{Fleming} is to use the powerful stability property of viscosity solutions in order to obtain a significant simplified approach. We refer to Feng and Kurtz \cite{FengKurtz} for a systematic application of this methodology with relevant extensions. 

The main objective of this paper is to extend the viscosity solutions approach to some problems of large deviations with rare events induced by non-Markov diffusions
 \bea
 \label{SDE}
 X_t
 \;=\;
 X_0+\int_0^t b_s(W,X)ds+\int_0^t\sigma_s(W,X)dW_s,
 &t\ge 0,&
 \eea
where $W$ is a Brownian motion, and $b,\sigma$ are non-anticipative functions of the paths of $(W,X)$ satisfying convenient conditions for existence and uniqueness of the solution of the last stochastic differential equation (SDE).

We should note that the Large Deviation Principle (LDP) for non-Markovian diffusions of  type (\ref{SDE}) is not
new. For example, Gao \& Liu \cite{GL} studied such a problem via the sample path LDP method by Fredlin-Wentzell, using various norms in infinite dimensional spaces. While the techniques there are quite deep and sophisticated, the methodology is more or less ``classical." Our main focus in this work is to extend the  
%is  and the main technique involves some 
%Our approach is to mimic the 
PDE approach of Fleming \cite{Fleming} in the present path-dependent framework, with a different set of tools. These include the theories of backward SDEs, stochastic control, and the viscosity solution for path-dependent PDEs (PPDEs), among them the last one has been developed
only very recently. Specifically, the theory of backward SDEs, pioneered by Pardoux \& Peng \cite{PardouxPeng}, can be
effectively used as a substitute to the partial differential equations in the Markovian setting. Indeed, the log-transformation of the vanishing probability solves a semilinear PDE in the Markovian case. However, due to the ``functional" nature of the coefficients in (\ref{SDE}), both backward SDE and PDE involved will become non-Markovian and/or path-dependent. 

Several technical points are worth mentioning. First, since the PDE involved in our problem naturally 
%leads to  one that 
has the nonlinearity in the gradient term (quadratic to be specific), we therefore need the extension
 by Kobylanski \cite{Kobylanski} on backward SDEs  to this context.
 %, see the review in Cvitani\'c \& Zhang \cite{CvitanicZhang}. 
 Second, in order to obtain the rate function, we exploit the stochastic control representation of the log-transformation, and proceed to the asymptotic analysis with crucial use of the BMO properties of the solution of the BSDE. Finally, we use the notion of viscosity solutions of path-dependent Hamilton-Jacobi equations introduced by Lukoyanov \cite{Lukoyanov} in order to characterize the rate function as unique viscosity solution of a path dependent Eikonal equation.

Another main purpose, in fact the original motivation, of this work is an application in financial mathematics. It has been
known that an important problem in the valuation and hedging of exotic options is to characterize the short time asymptotics of the implied volatility surface,
%We also provide an application of our results to which is in fact the initial motivation of the present work.  
given the prices of European options for all maturities and strikes. The need to resort to asymptotics is due to the fact that only a discrete set of maturities and strikes are available. This difficulty is bypassed by practitioners by using the asymptotics in order to extend the volatility surface to the un-observed regimes. We refer to Henry-Labord\`ere \cite{Henry-Labordere}. The results available in this literature have been restricted to the Markovian case, and our results  in a sense opens the door to a general non-Markovian, path-dependent paradigm.

We finally observe that the sequence of vanishing probabilities induced by non-Markov diffusions can be re-formulated in the Markov case by using the Gy\"{o}ngy's \cite{Gyongy} result which produces a Markov diffusion with the same marginals. However, the regularity of the coefficients of the resulting Markov diffusion $\sigma^X(t,x):=\dbE[\sigma_t|X_t=x]$ are in general not suitable for the application of the classical large deviation results. 

The paper is organized as follows. Section \ref{sect-formulation} contains the general setting, and provides our main results. First, we solve the small noise large deviation problem for the Laplace transform induced by a non-Markov diffusion. Next, we state the small noise large deviation result for the probability of exiting from some bounded open domain before some given maturity. We then state the characterization of the rate function as a unique viscosity solution of the corresponding path-dependent Eikonal equation. Section \ref{sect:application} is devoted to the application to the short maturity asymptotics of the implied volatility surface. Finally, Sections \ref{sect:Prob1}, \ref{sect:Prob2} and \ref{sect:existence} contain the proofs of our large deviation results, and the viscosity characterization. 

\section{Problem formulation and main results}
\label{sect-formulation}
\setcounter{equation}{0}

Let $\O_d:=\{\o\in C^0([0,T],\dbR^d):\o_0=0\}$ be the canonical space of continuous paths starting from the origin, $B$ the canonical process defined by $B_t:=\omega_t$, $t\in[0,1]$, and $\dbF:=\{\cF_t, t\in[0,T]\}$ the corresponding filtration. We shall use the following notation for the supremum norm:
 \beaa
 \|\omega\|_t:=\sup_{s\in[0,t]}|\omega_s|
 ~~\mbox{and}~~
  \|\omega\|:= \|\omega\|_T
 &\mbox{for all}&
 t\in[0,T],~\omega\in\Omega_d.
 \eeaa
Let $\dbP_0$ be the Wiener measure on $\Omega_d$. For all $\eps\ge 0$, we denote by $\dbP^\eps:=\dbP_0\circ(\sqrt{\eps}B)^{-1}$ the probability measure such that 
 \beaa
 &\big\{W^\e_t
 :=
 \frac{1}{\sqrt{\e}}B_t
 ,0\le t\le T\big\}&
 \mbox{is a}~\dbP^\eps-\mbox{Brownian motion.}
 \eeaa 
Our main interest in this paper is on the solution of the path-dependent stochastic differential equation:
\bea\label{sde}
 dX_t
 =
 b_t(B,X)dt+\sigma_t(B,X)dB_t,\ \ X_0=x_0,
 &&
 \dbP^\e\text{-a.s.}
 \eea
where the process $X$ takes values in $\dbR^n$ for some integer $n>1$, and its paths are in $\O_n:=C^0([0,T],\dbR^n)$. 

The supremum norm on $\O_n$ is also denoted $\|.\|_t$, without reference to the dimension of the underlying space. The coefficients $b:[0,T]\times\O_d\times\O_n\longrightarrow \dbR^{n}$ and $\sigma:[0,T]\times\O_d\times\O_n\longrightarrow \dbR^{n\times d}$ are assumed to satisfy the following conditions which guarantee existence and uniqueness of a strong solution for all $\eps>0$.

\begin{assum}\label{assum:laplace}  
The coefficients $f\in\{b,\sigma\}$ are:
\\
$\bullet$ non-anticipative, i.e. $f_t(\omega,x)=f_t\big((\omega_s)_{s\le t}, (x_s)_{s\le t}\big)$,
\\
$\bullet$ $L-$Lipschitz-continuous in $(\o,x)$, uniformly in $t$, for some $L>0$:
 $$
 \big|f_t(\o,x)-f_t(\o',x')\big|
 \leq
 L(\|\o-\o'\|_t+\|x-x'\|_t);
 ~t\in[0,T],(\omega,x),(\omega',x')\in\Omega_d\times\Omega_n,
 $$
\end{assum}

Under $\dbP^\eps$, the stochastic differential equation \eqref{sde} is driven by a small noise, and our objective is to provide some large deviation asymptotics in the present path-dependent case, which extend the corresponding results of Freidlin \& Wentzell \cite{FreidlinWentzell} in the Markovian case. Our objective is to adapt to our path-dependent case the PDE approach to large deviations of stochastic differential equation as initiated by Fleming \cite{Fleming} and Evans \& Ishii \cite{EvansIshii}, see also Fleming \& Soner \cite{FlemingSoner}, Chapter VII.

\subsection{Laplace transform near infinity}

As a first example, we consider the Laplace transform of some path-dependent random variable $\xi\big((\omega_s)_{s\le T},(x_s)_{s\le T}\big)$ for some final horizon $T>0$:
 \bea
 L^\e_0
 &:=& 
 -\eps \ln \dbE^{\dbP^\e}\Big[e^{-\frac{1}{\e}\xi(B,X)}\Big].
 \eea
In the following statement $\dbL^2_d$ denotes the collection of measurable functions $\alpha:[0,T]\longrightarrow\dbR^d$ such that $\int_0^T[\alpha_t|^2dt<\infty$.
Our first main result is:

\begin{thm}\label{thm:laplace}
Let $\xi$ be a bounded uniformly continuous $\cF_T-$measurable r.v. Then, under Assumption \ref{assum:laplace}, we have:
 \beaa
 L^\e_0
 \longrightarrow 
 L_0:=\inf_{\a\in\dbL^2_d} \ell_0^\a
 &\mbox{as}~\e\to 0,
 ~\mbox{where}&
 \ell_0^\a:=\xi(\o^\a,x^\a)+\frac12\int_0^T|\alpha_t|^2dt,
 \eeaa
and $(\o^\a,x^\a)$ are defined by the controlled ordinary differential equations:
 $$
 \o^\a_t
 =
 \int_0^t \a_sds
 ,~~
 x^\alpha_t
 =
 X_0+\int_0^t b_s(\o^\a,x^\a)ds+\int_0^t\sigma_s(\o^\a,x^\a)d\o^\a_s,
 ~~
 t\in[0,T].
 $$
\end{thm}

The proof of this result is reported in Section \ref{sect:Prob1}.

\begin{rem}\label{rem:beps}{\rm 
Theorem \ref{thm:laplace} is still valid in the context where the coefficient $b$ depends also on the parameter $\eps$, so that the process $X$ is replaced by $X^\eps$ defined by:
 \beaa
 dX^\eps_t
 =
 b^\eps_t(B,X^\eps)dt+\sigma_t(B,X^\eps)dB_t,\ \ X^\eps_0=x_0,
 &&
 \dbP^\e\text{-a.s.}
 \eeaa
Since this extension will be needed for our application in Section \ref{sect:application}, we provide a precise formulation. Let Assumption \ref{assum:laplace} hold uniformly in $\eps\in [0,1)$, and assume further that $\eps\longmapsto b^\eps$ is uniformly Lipschitz on $[0,1)$. Then the statement of Theorem \ref{thm:laplace} holds with $x^\alpha$ defined by:
 \beaa
 x^\alpha_t
 =
 X_0+\int_0^t b^0_s(\o^\a,x^\a)ds+\int_0^t\sigma_s(\o^\a,x^\a)d\o^\a_s,
 &
 t\in[0,T].&
 \eeaa
This slight extension does not induce any additional technical difficulty in the proof. We shall therefore provide the proof in the context of Theorem \ref{thm:laplace}. 
 }
\end{rem}

\subsection{Exiting from a given domain before some maturity}
\label{subsect:exit}
 
As a second example, we consider the asymptotic behavior of the probability of exiting from some given subset of $\dbR^n$ before the maturity $T$: 
 \bea
 Q^\e_0
 :=
 -\e\ln\dbP^\e[H<T],
 &\mbox{where}&
 H:=\inf\{t>0:X_t\notin O\},
 \eea 
and $O$ is a bounded open set in $\dbR^n$. We also introduce the corresponding subset of paths in $\O_n$:
 \bea
 \cO
 &:=&
 \big\{\o\in\O:\o_t\in O\text{ for all }t\le T\big\}.
 \eea
The analysis of this problem requires additional conditions.

\begin{assum}\label{assum:hitting}
The coefficients $b$ and $\si$ are uniformly bounded, and $\sigma$ is uniformly elliptic, i.e. $a:=\sigma\sigma^{\rm T}$ is invertible with bounded inverse $a^{-1}$.
\end{assum}

The present example exhibits a singularity on the boundary $\partial O$ because $Q^\eps_0$ vanishes whenever the path $\omega$ is started on the boundary $\partial O$. Our second main result is the following.

\begin{thm}\label{thm:exit}
Let $O$ be a bounded open set in $\dbR^n$ with $C^3$ boundary. Then, under Assumptions \ref{assum:laplace} and \ref{assum:hitting}, we have:
 \beaa
 Q^\e_0 
 \longrightarrow Q_0
 :=
 \inf\big\{q_0^\alpha: \a\in\dbL^2_d,~x^\a_{T\wedge\cdot}\notin\cO\big\},
 &\mbox{where}&
 q_0^\alpha:=\frac12 \int_0^T|\a_s|^2ds,
 \eeaa
and $x^\a$ is defined as in Theorem \ref{thm:laplace}.
\end{thm}

The proof of this result is reported in Section \ref{sect:Prob2}.

\begin{rem}\label{rem:beps2}
{\rm
(i) A similar result of Theorem \ref{thm:exit} can be found in Gao-Liu \cite{GL}. However, our proof has a 
completely different flavor and, given the preparation of the PPDE theory, seems to be more direct, whence shorter.

(ii) The condition on the boundary $\partial O$ can be slightly weakened. Examining the proof of Lemma \ref{lem:upbound}, where this condition is used, we see that it is sufficient to assume that $O$ can be approximated from outside by open bounded sets with $C^3$ boundary.
}
\end{rem}

\begin{rem}\label{rem:beps2}
{\rm
The result of Theorem \ref{thm:exit} is still valid in the context of Remark \ref{rem:beps}. This can be immediately verified by examining the proof of Theorem \ref{thm:exit}.
}
\end{rem}

\subsection{Path-dependent Eikonal equation}

We next provide a characterization of our asymptotics in terms of partial differential equations. We refer to Evans \& Ishii \cite{EvansIshii}, Fleming \& Souganidis \cite{FlemingSouganidis}, Evans-Souganidis \cite{EvansSouganidis}, Evans, Souganidis, Fournier \& Willem \cite{EvansSouganidisFournierWillem}, Fleming \& Soner \cite{FlemingSoner}, for the corresponding PDE literature with a derivation by means of the powerful theory of viscosity solutions. 

Due to the path dependence in the dynamics of our state process $X$, and the corresponding limiting system $x^\alpha$, our framework is clearly not covered by any of these existing works. Therefore, we shall adapt the notion of viscosity solutions introduced in Lukoyanov \cite{Lukoyanov}. 

Consider the truncated Eikonal equation:
 \bea\label{PPDE}
 \big\{-\partial_tu-F_{K_0}\big(.,\partial_\omega u,\partial_x u\big)\big\}(t,\omega,x)
 =0
 &\mbox{for}&
 (t,\o,x)\in\Theta^0,
 \eea
where $K_0$ is a fixed parameter, and the nonlinearity $F_{K_0}$ is given by:
 \bea\label{def:F}
 F_{K_0}(\theta,p_\omega,p_x)
 :=
 b(\theta)\cdot p_x
 +\inf_{|a|\leq K_0}\Big\{\frac12 a^2 + a\big(p_\omega+\sigma(\theta)^{\rm T}p_x\big)\Big\},
\eea
for all $\th\in\Th$, $p_\o\in\dbR^d$ and $p_x\in\dbR^n$. Notice that
 \beaa
 F_{K_0}(\theta,p_\omega,p_x)
 \longrightarrow
 b(\th)\cdot p_x 
 -\frac12\big|p_\o+\sigma^{\rm T}p_x\big|^2
 &\mbox{as}&
 K_0\to\infty,
 \eeaa
the equation (\ref{PPDE}) thus leads to a path-dependent Eikonal equation. We note that the truncated feature of the equation \eqref{PPDE} is induced by the fact that the corresponding solution will be shown to be Lipschitz under our assumptions.

\subsubsection{Classical derivatives}

Denote $\hat\Omega:=\O_d\times\O_n$ and $\hat\omega=(\omega,x)$ a generic element of $\hat\Omega$, $\Theta:=[0,T]\times\hat\O$, and $\Theta^0:=[0,T)\times\hat\O$. The set $\Theta$ is endowed with the pseudo-distance
 \beaa
 d(\theta,\theta')
 :=
 |t-t'|
 +\big\|\hat\omega_{t\wedge}-\hat\omega'_{t'\wedge}\big\|
 &\mbox{for all}&
 \theta=(t,\hat\omega),\theta'=(t',\hat\omega')\in\Theta.
 \eeaa
For any integer $k>0$, we denote by $C^0(\Theta,\dbR^k)$ the collection of all continuous function $u:\Theta\longrightarrow\dbR^k$. Notice, in particular, that any $u\in C^0(\Theta,\dbR^k)$ is non-anticipative, i.e. $u(t,\hat\omega)=u(t,(\hat\omega_s)_{s\le t})$ for all $(t,\hat\omega)\in\Theta$.

We denote $\hat\O_K$ as the set of all $K$-Lipschitz paths. For $\theta=(t,\hat\omega)\in\Theta^0$, we denote $\Theta(\theta):=\cup_{K\ge 0}\Theta_K(\theta)$, where:
 \beaa
 \Theta_K(\theta)
 &:=&
 \big\{(t',\hat\omega')\in\Theta:
 t'\ge t,
 ~\hat\omega'_{t\wedge}=\hat\omega_{t\wedge},
 ~\mbox{and}~
 \hat\omega' |_{[t,T]}~\mbox{is $K-$Lipschitz}
 \big\}.
 \eeaa

\begin{defn}
A function $\varphi:\Theta\longrightarrow\dbR$ is said to be $C^{1,1}(\Theta)$ if $\varphi\in C^0(\Theta,\dbR)$, and we may find $\partial_t\varphi\in C^0(\Theta,\dbR)$, $\partial_{\hat\omega} \varphi\in C^0(\Theta,\dbR^{d+n})$, such that for all $\theta=(t,\hat\omega)\in\Theta$:
 \beaa
 \varphi(\theta')
 =
 \varphi(\theta)
 + \partial_t\varphi(\theta)(t'-t)
 + \partial_{\hat\omega} \varphi(\theta)(\hat\omega'_{t'}-\hat\omega_t)
 + \circ_{\hat\omega'}(t'-t)
 &\mbox{for all}&
 \theta'\in\Theta(\theta),
 \eeaa
where $\circ_{\hat\omega'}(h)/h\longrightarrow 0$ as $h\searrow 0$. The derivatives $\partial_\o$ and $\partial_x$ are defined by the natural decomposition $\partial_{\hat\o}\varphi=(\partial_\o\varphi,\partial_x\varphi)^{\rm T}$.
\end{defn}

The last collection of smooth functions will be used for our subsequent definition of viscosity solutions. 

\subsubsection{Viscosity solutions of the path-dependent Eikonal equation}

Let $\Th^0_K:=[0,T)\times\hat\O_K$. The set of test functions is defined for all $K>0$ and $\th\in\Th^0_K$ by:
 \bea
 \underline{\cA}^K u(\theta)
 &:=&
 \big\{\f\in C^{1,1}(\Theta):
         (\f-u)(\theta)=\min_{\theta'\in\Theta_K}(\f-u)(\theta')
 \big\},
 \\
 \overline{\cA}^K u(\theta)
 &:=&
 \big\{\f\in C^{1,1}(\Theta):
         (\f-u)(\theta)=\max_{\theta'\in\Theta_K}(\f-u)(\theta')
 \big\}.
 \eea

\begin{defn}
Let $u:\Theta\longrightarrow\dbR$ be a continuous function.
\\
{\rm (i)} $u$ is a $K$-viscosity subsolution of \eqref{PPDE}, if for all $\theta\in\Th^0_K$, we have 
 \beaa
 \big\{ -\pa_t \f
          -F_{K_0}(.,\pa_{\hat\omega} \f)
 \big\}(\theta) 
 \le 0
 &\mbox{for all}&
  \f\in \underline{\cA}^K u(\th).
  \eeaa
{\rm (ii)} $u$ is a $K$-viscosity supersolution of \eqref{PPDE}, if for all $\theta\in\Th^0_K$, we have
 \beaa
 \big\{ -\pa_t\f
          -F_{K_0}(.,\pa_{\hat\omega} \f)
 \big\}(\theta) 
 \ge 0
 &\mbox{for all}&
  \f\in \overline{\cA}^K u(\th).
  \eeaa
{\rm (iii)} $u$ is a $K$-viscosity solution of \eqref{PPDE} if it is both $K$-viscosity  subsolution and supersolution.
\end{defn}

\subsubsection{Wellposedness of the path-dependent Eikonal equation}

We only focus on the asymptotics of Laplace transform. For simplicity, we adopt the following strengthened version
of Assumption  \ref{assum:laplace}.

\begin{assum}\label{assum:ppde}
The coefficients $b$ and $\sigma$ are bounded and satisfy Assumption \ref{assum:laplace}. 
\end{assum}

A natural candidate solution of equation \eqref{PPDE} is the dynamic version of the limit $L^0$ introduced in Theorem \ref{thm:laplace}:
 \bea\label{potential_sol}
 u(t,\hat\o)
 &:=&
 \inf_{\a\in\dbL^2_d([t,T])} \Big\{\xi^{t,\hat\o}(\hat\o^{\a,t,\hat\o})+\frac12\int_t^{T}|\a_s|^2 ds\Big\},
 ~~(t,\hat\o)\in\Theta,
 \eea
where $\hat\o^{\a,t,\hat\o}:=(\o^{\a,t,\hat\o},x^{\a,t,\hat\o})$ is defined by:
 $$
 \o_s^{\a,t,\hat\o}
 =
 \int_0^s \a_{t+r}dr,
 ~~x_s^{\a,t,\hat\o}
 =
 \int_0^s b_{t+r}(\hat\o\otimes_t\hat\o^{\a,t,\hat\o})dr
 +\int_0^s \si_{t+r}(\hat\o\otimes_t\hat\o^{\a,t,\hat\o})d\o_r^{\a,t,\hat\o},
 $$
with the notation $(\hat\o\otimes_t\hat\o')_s:=\1_{\{s\le t\}}\hat\o_s+\1_{\{s>t\}}\big(\hat\o_t+\hat\o'_{s-t}\big)$, and
 \beaa
 \xi^{t,\hat\o}(\hat\o')
 :=
 \xi\big((\hat\o\otimes_t\hat\o')_{T\wedge\cdot}\big)
 &\mbox{for all}&
 \hat\o,\hat\o'\in\hat\O.
 \eeaa

\begin{thm}\label{thm: viscosity sol}
Let Assumption \ref{assum:ppde} hold true, and let $\xi$ be a bounded Lipschitz function on $\hat\O$. Then, for $K$ and $K_0$ sufficiently large, the function $u$ defined in \eqref{potential_sol} is the unique bounded $K$-viscosity solution of the path-dependent PDE \eqref{PPDE}.
\end{thm}

The proof of this result is reported in Section \ref{sect:existence}.

\section{Application to implied volatility asymptotics}
\label{sect:application}

\subsection{Implied volatility surface}

The Black-Scholes formula $\mbox{BS}(K,\si^2T)$ expresses the price of a European call option with time to maturity $T$ and strike $K$ in the context of a geometric Brownian motion model for the underlying stock, with volatility parameter $\sigma\ge 0$:
 \beaa
 \widehat{\mbox{BS}}(k,v)
 \;:=\;
 \frac{\mbox{BS}(K,v)}{S_0}
 &:=&
 \left\{\begin{array}{l}
         (1-e^k)^+~~\mbox{for}~~v=0,
         \\
         \mathbf{N}\big(d_+(k,v)\big)-e^{k}\mathbf{N}\big(d_-(k,v)\big),
         ~~\mbox{for}~~v>0,
         \end{array}
 \right.
 \eeaa
where $S_0$ denotes the spot price of the underlying asset, $v:=\sigma^2 T$ is the total variance, $k:=\ln(K/S_0)$ is the log-moneyness of the call option, $\mathbf{N}(x):=(2\pi)^{-1/2}\int_{-\infty}^xe^{-y^2/2}dy$,
 \beaa
 d_\pm(k,v):=\frac{-k}{\sqrt{v}}\pm\frac{\sqrt{v}}{2},
 \eeaa
and the interest rate is reduced to zero.

We assume that the underlying asset price process is defined by the following dynamics under the risk-neutral measure $\dbP_0$:
 \beaa
 dS_t
 =
 S_t\sigma_t(B,S)dB_t,
 &\dbP_0-\mbox{a.s.}&
 \eeaa
so that the price of the $T-$maturity European call option with strike $K$ is given by $\dbE^{\dbP_0}\big[(S_T-K)^+\big]$. The implied volatility surface $(T,k)\longmapsto\Sigma(T,k)$ is then defined as the unique non-negative solution of the equation
 \beaa
 \mathbf{N}\big(d_+(k,\Sigma^2T)\big)
             -e^{k}\mathbf{N}\big(d_-(k,\Sigma^2T)\big)
 &=&
 \hat C(T,k)
 \;:=\;
 \dbE^{\dbP_0}\big[\big(e^{X_T}-e^{k}\big)^+\big],
 \eeaa
where $X_t:=\ln{(S_t/S_0)}$, $t\ge 0$. 

Our interest in this section is on the short maturity asymptotics $T\searrow 0$ of the implied volatility surface $\Sigma(T,k)$ for $k>0$. This is a relevant practical problem which is widely used by derivatives traders, and has induced an extensive literature initiated by Berestycki, Busca \& Florent \cite{BerestyckiBuscaFlorent1,BerestyckiBuscaFlorent2}. See e.g. Henry-Labord\`ere \cite{Henry-Labordere}, Hagan, Lesniewski, \& Woodward \cite{HaganLesniewskiWoodward}, Ford and Jacquier \cite{FordJacquier}, Gatheral, Hsu, Laurence, Ouyang \& Wang \cite{GatheralHsuLaurenceOuyang}, Deuschel, Friz, Jacquier \& Violante \cite{DeuschelFrizJacquierViolante1,DeuschelFrizJacquierViolante2}, and Demarco \& Friz \cite{DemarcoFriz}.

Our starting point is the following limiting result which follows from standard calculus:
 \beaa
 \lim_{v\rightarrow 0} v\ln{\widehat{\mbox{BS}}(k,v)}
 \;=\;
 -\frac{k^2}{2},
 &\mbox{for all}&
 k>0.
 \eeaa
We also compute directly that, for $k>0$, we have $\hat C(T,k)\longrightarrow 0$ as $T\searrow 0$. Then $T\Sigma(T,k)^2 \longrightarrow 0$ as $T\searrow 0$, and it follows from the previous limiting result that
 \bea\label{impliedvol-asymptotics}
 \lim_{T\rightarrow 0}T\Sigma(T,k)^2\ln\hat C(T,k)
 \;=\;
 -\frac{k^2}{2},
 &\mbox{for all}&
 k>0.
 \eea
Consequently, in order to study the asymptotic behavior of the implied volatility surface $\Sigma(T,k)$ for small maturity $T$, we are reduced to the asymptotics of $T\ln\hat C(T,k)$ for small $T$, which will be shown in the next subsection to be closely related to the large deviation problem of Subsection \ref{subsect:exit}. Hence, our path-dependent large deviation results enable us to obtain the short maturity asymptotics of the implied volatility surface in the context where the underlying asset is a non-Markovian martingale under the risk-neutral measure. 

\subsection{Short maturity asymptotics}

Recall the process $X_t:=\ln (S_t/S_0)$. By It\^o's formula, we deduce the dynamic for process $X$:
 \be\label{SDE-finance}
 dX_t
 \;=\;
 -\frac12 \si^X_t(B,X)^2 d\langle B\rangle_t
 +\si^X_t(B,X)dB_t,
 \ee
where $\si^X(\o,x):=\si\big(\o,S_0e^{x_\cdot}\big)$. For the purpose of the application in this section, we need to convert the short maturity asymptotics into a small noise problem, so as to apply the main results from the previous section. In the present path-dependent case, this requires to impose a special structure on the coefficients of the stochastic differential equation \eqref{SDE-finance}. 

For a random variable $Y$ and a probability measure $\dbP$, we denote by $\cL^{\dbP}(Y)$ the $\dbP-$distribution of $Y$.

\begin{assum}\label{assum:application}
The diffusion coefficient $\si^X:[0,T]\times\O_d\times\O_n\longrightarrow\dbR$ is non-anticipative, Lipschitz-continuous, takes values in $[\ul\si,\ol\si]$ for some $\ol\si\geq \ul\si>0$, and satisfies the following small-maturity small-noise correspondence:
  \beaa
  \cL^{\dbP_0}(X_\eps)
  \;=\;
  \cL^{\dbP^\eps}(X_1)
  &\mbox{for all}&
  \eps\in[0,1).
  \eeaa
\end{assum}

\begin{rem}{\rm 
Assume that $\sigma$ is independent of $\o$ and satisfies the following time-indifference property:
 \beaa
 \si^X_{ct}(x)
 =
 \si^X_t(x^c)
 ~~\mbox{for all}~~c>0,
 &\mbox{where}&
 x^c_s:=x_{cs},~~s\in[0,T].
 \eeaa
Then, $\cL^{\dbP_0}\big((X_s)_{s\le\eps}\big)=\cL^{\dbP^\eps}\big((X_s)_{s\le 1}\big)$ for all $\eps\in[0,1)$, which implies that the small-maturity small-noise correspondence holds true. In particular, the time-indifference property holds in the homogeneous Markovian case $\sigma_t(x)=\sigma(x_t)$.
}
\end{rem}

In view of \eqref{impliedvol-asymptotics} and the small-maturity small-noise correspondence of Assumption \ref{assum:application}, we are reduced  to the asymptotics of
 \beaa
 \e\ln \dbE^{\dbP^\e}[(e^{X_1}-e^k)^+]
 &\mbox{as}&
 \e\rightarrow 0.
 \eeaa
Under $\dbP^\e$ the dynamics of $X$ is given by the stochastic differential equation:
 \beaa
 dX_t
 \;=\;
 -\frac{\e}{2}\; \si^X_t(B,X)^2 dt
 +\si^X_t(B,X)dB_t,
 &\dbP^\e-\mbox{a.s.}&
 \eeaa
whose coefficients satisfy the conditions given in Remarks \ref{rem:beps} and \ref{rem:beps2}. Consider the stopping time 
 \beaa
 H_{a,b}
 \;:=\;
 \inf\{t: X_t\not\in (a,b)\}
 &\mbox{for}&
 -\infty<a<b<+\infty.
 \eeaa
Then, it follows from Theorem \ref{thm:exit} and Remark \ref{rem:beps2} that
 \beaa
 Q^\eps_0:=-\eps\ln{\dbP^\eps\big[H_{a,b}\leq 1\big]}
 \longrightarrow
 Q_0(a,b)
 &\mbox{as}&
 \eps\searrow 0,
 \eeaa
where $Q_0(a,b)$ is defined as in Theorem \ref{thm:exit} in terms of the controlled function $x^\a$ of Theorem \ref{thm:laplace}:
 \beaa
 Q_0(a,b)
 :=
 \inf\Big\{\frac12 \int_0^1|\a_s|^2ds: \a\in\dbL^2_d,~x^\a_{1\wedge\cdot}\notin \cO_{a,b}\Big\},
 \eeaa
 where $\cO_{a,b}:=\big\{x:x_t\in ( a,  b)~\mbox{for all}~t\in[0,1]\big\}$. The rest of this section is devoted to the following result.

\begin{prop}
$\lim_{\e\rightarrow 0}-\e\ln \dbE^{\dbP^\e}[(e^{X_1}-e^k)^+]=Q_0(k):=\lim_{a\rightarrow -\infty}Q_0(a,k)$.
\end{prop}

\begin{proof}
\no {\bf 1.}\quad We first show that
 \bea\label{limsup_app}
 \limsup_{\e\rightarrow 0}\e\ln \dbE^{\dbP^\e}[(e^{X_1}-e^k)^+]
 &\le&
 \;-Q_0(k).
\eea
Fix some $p>1$ and the corresponding conjugate $q>1$ defined by $\frac{1}{p}+\frac{1}{q}=1$. By the H\"{o}lder inequality, we estimate that
 $$
 \dbE^{\dbP^\e}\big[(e^{X_1}-e^k)^+\big]
 \le
 \dbE^{\dbP^\e}\Big[e^{X_1}\1_{\{X_1\geq k\}}\Big]
 \le 
 \dbE^{\dbP^\eps}\big[e^{qX_1}\big]^{1/q}\dbP^\e[H_{a,k}\leq 1]^{1/p},
 ~\mbox{for all}~a<k. 
 $$
By standard estimates, we may find a constant $C_p$ such that $\dbE^{\dbP^\e}\big[e^{qX_1}\big]
 \le 
 C_p$ for all $\e\in(0,1)$. Then, 
 \beaa
	\e \ln \dbE^{\dbP^\e}\big[(e^{X_1}-e^k)^+\big]
	&\le&
    \frac{\e}{q} \ln{C_p} + \frac{\e}{p}\ln\dbP^\e[H_{a,k}\leq 1],
 \eeaa
which provides \eqref{limsup_app} by sending $\e\rightarrow 0$ and then $p\rightarrow 1$.

\no {\bf 2.}\quad We next prove the following inequality: 
 \bea\label{liminf_app}
 \liminf_{\e\rightarrow 0}\e\ln \dbE^{\dbP^\e}[(e^{X_1}-e^k)^+]
 &\ge& 
 -Q_0(k).
 \eea
For $n\in\dbN$, denote $f_n(x):=(e^{-n}-x)^+ +(x-e^k)^+$ for $x\in\dbR$.
Since $f_n$ is convex and $e^X$ is $\dbP^\e$-martingale, the process $f\big(e^X\big)$ is a non-negative $\dbP^\e$-submartingale. For a sufficiently small $\d>0$, set $a_{n,\d}:=\ln(e^{-n}-\d)$ and $k_\d:=\ln(e^k+\d)$. Then, it follows from the Doob inequality that 
 \be\label{Doobineq}
 \dbP^\e[H_{a_{n,\d},k_\d}\leq 1]
 \;=\; 
 \dbP^\e\Big[\max_{t\le 1}f_n\big(e^{X_t}\big)\ge\d\Big]
 \;\le\; 
 \frac{1}{\d}\;\dbE^{\dbP^\e}\big[f_n\big(e^{X_1}\big)\big].
 \ee
We shall prove in Step 3 below that
 \bea\label{limsupbounded}
 \lim_{\e\rightarrow 0}
 \frac{\dbE^{\dbP^\e}[(e^{-n}-e^{X_1})^+]}
        {\dbE^{\dbP^\e}[(e^{X_1}-e^k)^+]}
 =0
 &\mbox{for large}&
 n.
 \eea
Then, it follows from \eqref{Doobineq}, by sending $\e\rightarrow 0$, that
 \beaa
 -Q_0(a_{n,\d},k_\d)
 &\leq&
 \liminf_{\e\rightarrow 0} \e\ln\dbE^{\dbP^\e}[(e^{X_1}-e^k)^+].
 \eeaa
Finally, sending $\d\rightarrow 0$ and then $n\rightarrow\infty$, we obtain (\ref{liminf_app}).

\no {\bf 3.}\quad It remains to prove \eqref{limsupbounded}. Since $\ul\si \leq \si \leq \ol\si$, by Assumption \ref{assum:application}, it follows from the convexity of $s\longmapsto (e^{-n}-s)^+$ and $s\longmapsto (s-e^k)^+$ that
 \beaa
 \frac{\dbE^{\dbP^\e}[(e^{-n}-e^{X_1})^+]}{\dbE^{\dbP^\e}[(e^{X_1}-e^k)^+]}
 \;\leq\;
 \frac{\dbE^{\dbP^\e}[(e^{-n}-e^{-\frac12 \e\ol\si^2+\ol\si B_1})^+]}
        {\dbE^{\dbP^\e}[(e^{-\frac12 \e\ul\si^2+\ul\si B_1}-e^k)^+]}.
\eeaa
Further, we have
 \beaa
 \dbE^{\dbP^\e}\big[\big(e^{-n}-e^{-\frac12 \e\ol\si^2+\ol\si B_1}\big)^+\big]
 &\le& 
 e^{-n}\mathbf{N}\Big(\frac12\ol\si\sqrt\e-\frac{n}{\ol\si\sqrt\e}\Big),
 \eeaa
and, by the Chebyshev inequality,
 $$
 \dbE^{\dbP^\e}[(e^{-\frac12 \e\ul\si^2+\ul\si B_1}-e^k)^+]
 \geq 
 \l\dbP^\e[e^{-\frac12 \e\ul\si^2+\ul\si B_1}\geq e^k+\l]\\
 =
 \l\mathbf{N}\Big(-\frac12\ul\si\sqrt\e-\frac{\ln(e^k+\l)}{\ul\si\sqrt\e}\Big).
 $$
Using the estimate $\mathbf{N}(-x)\sim\frac{1}{\sqrt{2\pi}}x^{-1}e^{-\frac{x^2}{2}}$, we obtain that
 \beaa
 \limsup_{\e\rightarrow 0}
 \frac{\dbE^{\dbP^\e}[(e^{-n}-e^{X_1})^+]}
        {\dbE^{\dbP^\e}[(e^{X_1}-e^k)^+]}
 &\leq&
 C \exp\Big\{-\lim_{\e\rightarrow 0} \frac{1}{2\e}
                                                     \Big(\frac{n^2}{\ol\si^2}-\frac{(\ln(e^k+\l))^2}{\ul\si^2}\Big)
           \Big\}
 \;=\;
 0,
\eeaa
whenever $n^2>\frac{\ol\si^2}{\ul\si^2}(\ln(e^k+\l))^2$.  
\end{proof}

\section{Asymptotics of Laplace transforms}
\label{sect:Prob1}

Our starting point is a characterization of $Y^\e_0$ in terms of a quadratic backward stochastic differential equation. Let
 \bea
 Y^\e_t:=-\eps \ln \dbE^{\dbP^\e}_t\Big[e^{-\frac{1}{\e}\xi(B,X)}\Big],
 &t\in[0,T].&
 \eea
where $\dbE^{\dbP^\e}_t$ denotes expectation operator under $\dbP^\e$, conditional to $\cF_t$.

\begin{prop}\label{bsde}
The processes $Y^\eps$ is bounded by $\|\xi\|_\infty$, and is uniquely defined as the bounded solution of the quadratic backward stochastic differential equation
 \beaa
 Y^\e_t
 &=&
 \xi-\frac12\int_t^T\big|Z^\e_s\big|^2ds
 +\int_t^T Z^\e_s \cdot dB_s,
 ~~\dbP^\e-\mbox{a.s.}
 \eeaa
Moreover, the process $Z^\e$ satisfies the BMO estimate
 \bea
 \|Z\|_{\Hbmo^2(\dbP^\e)}
 :=
 \sup_{t\in[0,T]}
 \Big\|
 \dbE^{\dbP^\e}_t\int_t^T \big|Z^\e_s\big|^2ds
 \Big\|_{\dbL^\infty(\dbP^\eps)}
 &\le&
 4\|\xi\|_\infty.
 \eea
\end{prop}

\proof Since $\xi$ is bounded, we see immediately that $Y^\e_t\le-\e\ln\big(e^{-\frac{1}{\e}\|\xi\|_\infty}\big)=\|\xi\|_\infty$ and, similarly $Y^\e_t \ge-\|\xi\|_\infty$. Consequently, the process 
$$p^\e:=e^{-\frac{1}{\e}Y^\e}=\dbE_t^{\dbP^\e}[e^{-\frac{1}{\e}\xi(B,X)}]$$ 
is a bounded martingale. By martingale representation, there exists a process $q^\e$, with $\dbE^{\dbP^{\eps}}\big[\int_0^T|q^\eps_t|^2dt\big]<\infty$, such that $p^\e_t=p^\e_0+\int_0^tq^\e_s\cdot dB_s,$ for all $t\in[0,T]$.
Then, $Y^\e$ solves the quadratic backward SDE by It\^o's formula. The estimate $\|Z\|_{\dbH^2_{\mbox{\sc\tiny bmo}}(\dbP^\e)}$ follows immediately by taking expectations in the quadratic backward SDE, and using the boundedness of $Y^\e$ by $\|\xi\|_\infty$.
\qed
 
\vspace{5mm}

We next provide a stochastic control representation for the process $Y^\e$. For all $\alpha\in\Hbmo^2$, we introduce
 \beaa
 M^{\e,\a}_T
 &:=&
 e^{\frac{1}{\e}\int_0^T\a_t\cdot dB_t
       -\frac{1}{2\e}\int_0^T|\a_t|^2 dt.
       }
 \eeaa
Then $\dbE^{\dbP^\e}\big[M^{\e,\a}_T\big]=1$, and we may introduce an equivalent probability measure $\dbP^{\e,\a}$ by the density $d\dbP^{\e,\a}:=M^{\e,\a}_T d\dbP^\e$. Define:
 \beaa
 Y^{\e,\a}_t
 &=&
 \dbE^{\dbP^{\e,\a}}\Big[\xi+\frac12\int_t^T|\alpha_s|^2ds\Big],
 ~~\dbP^\e-\mbox{a.s.}
 \eeaa

\begin{lem}\label{lem-stochrep}
We have
 \beaa
 Y^\e_0
 \;=\;
 Y^{\e,Z^\e}_0
 &=&
 \inf_{\a\in\Hbmo^2(\dbP^\e)}Y^{\e,\a}_0.
 \eeaa
\end{lem}

\proof Notice that $Y^{\e,\a}$ solves the linear backward SDE
 \beaa
 dY^{\e,\a}_t
 &=&
 -Z^{\e,\a}_t\cdot dB_t- \big(Z^{\e,\a}_t\cdot \a_t-\frac{1}{2}|\a_t|^2\big)dt,
 ~~\dbP^\e-\mbox{a.s.}
 \eeaa
Since $-\frac12 z^2=\inf_{a\in\dbR^d}\big\{-a\cdot z+\frac12 a^2\big\}$, it follows from the comparison of BSDEs that $Y^{\e,\a}\ge Y^\e$. The required result follows from the observation that the last supremum is attained by $a^*=z$, and that $Y^{\e,Z^\e}=Y^\e$.
\qed

\vspace{5mm}

\no {\bf Proof of Theorem \ref{thm:laplace}.}\quad First, it is clear that $ \dbL^2_d \subset \cap_{\e>0} \Hbmo^2(\dbP^\e)$. Let $\alpha\in \dbL^2_d$ and any $\e>0$ be fixed. Since $\alpha$ is deterministic, it follows from the Girsanov Theorem that 
\beaa
Y^{\e,\a}_0 = \dbE^{\dbP_0}\Big[\xi(W^{\e,\a}, X^{\e,\a}) + {1\over 2} \int_0^T |\a_t|^2dt\Big],
\eeaa
where
\beaa
\left.\ba{lll}
W^{\e,\a}_t &:=&  \sqrt{\e} B_t + \int_0^t \a_s ds,\\
 X^{\e,\a}_t &=& X_0 + \int_0^t b_s(W^{\e,\a}, X^{\e,\a}_s) ds +  \int_0^t \si_s(W^{\e,\a}, X^{\e,\a}_s) d W^{\e,\a}_s,
 \ea\right.
 \dbP_0\mbox{-a.s.}
\eeaa
By the given regularities, it is clear that $\lim_{\e\to 0} Y^{\e,\a}_0 = l^\a_0$.  Then it follows from Lemma \ref{lem-stochrep} that
 \beaa
 \limsup_{\e\to 0}
 Y^\e_0
\;\le\;
 \limsup_{\e\to 0}
 Y^{\e,\a}_0
 &=&
 \ell_0^\a.
 \eeaa
By the arbitrariness of $\alpha\in\dbL^2_d$, this shows that $\limsup_{\e\to 0}Y^\e_0\le L_0$.

To prove the reverse inequality, we use the minimizer from Lemma \ref{lem-stochrep}. Note that $\dbP^\e$ is equivalent to $\dbP^{\e, Z^\e}$ and for $\dbP^\e$-a.e. $\o$, $\a^{\e,\o} := Z^\e_\cd (\o) \in \dbL^2_d$.  Then we compute that
 \beaa
 Y^\e_0
 \;=\;
 Y^{\e,Z^\e}_0
 &=&
 \dbE^{\dbP^{\e,Z^\e}}\Big[\xi(B,X)+\frac12\int_0^T\big|Z^\e_t\big|^2dt\Big]
 \\
 &\ge&
 L_0 +\dbE^{\dbP^{\e,Z^\e}}\Big[\xi(B,X)-\xi\big(\o^{Z^\e(\o)},x^{Z^\e(\o)}(\o)\big)\Big]
 \\
 &\ge&
 L_0 -\dbE^{\dbP^{\e,Z^\e}}\Big[\rho\big(\big\|B-\o^{Z^\e(\o)}\big\|_T
                                                                +\big\|X-x^{Z^\e(\o)}(\o)\big\|_T
                                                         \big)
                                          \Big].
 \eeaa
By definition of $\o^\a$, notice that $\o\longmapsto W^\e(\o):=\eps^{-1/2}\big(B(\o)-\o^{Z^\e(\o)}\big)$ defines a Brownian motion under $\dbP^{\e,Z^\e}$. Then it is clear that
\beaa
\limsup_{\e\to 0} \dbE^{\dbP^{\e,Z^\e}}\Big[\big\|B-\o^{Z^\e(\o)}\big\|_T\Big] =\limsup_{\e\to 0} \dbE^{\dbP^{\e,Z^\e}}\Big[\sqrt{\e}\|W^\e\|_T\Big]  = 0.
\eeaa

%Similarly, denoting by $\delta X:=X-x^{Z^\e(\o)}(\o)$. R
Furthermore, recall that $\si$ and $b$ are Lipschitz-continuous, it follows from the comparison of SDEs that $\underline\delta_t\le X-x^{Z^\e} \le \overline\delta_t$, where $\underline\delta_0=\overline\delta_0=0$, and 
 \beaa
 d\underline\delta_t
 &=&
 \sigma_t(B,X)\sqrt{\e}dW^\e_t
 -L\big(\sqrt{\e}\|W^\e\|_t+\|\underline\delta\|_t\big)\left(|Z^\e_t|+1\right)dt,
 \\
 d\overline\delta_t
 &=&
 \sigma_t(B,X)\sqrt{\e}dW^\e_t
 +L\big(\sqrt{\e}\|W^\e\|_t+\|\overline\delta\|_t\big)\left(|Z^\e_t|+1)\right)dt.
 \eeaa
We now estimate $\overline{\delta}$. The estimation of $\underline\delta$ follows the same line of argument. Denote $K_t:=\int_0^t\sigma_s(B,X)dW^\e_s$. By Gronwall's inequality, we obtain
  \beaa
  \eps^{-1/2}\|\overline{\d}_T\|
  &=&
  L\|W^\e\|_T\int_0^T e^{L\int_t^T\left(|Z^\e_s|+1\right)ds}\left(|Z^\e_t|+1\right) dt +\int_0^T e^{L\int_t^T\left(|Z^\e_s|+1\right)ds}d\|K\|_t
  \\
  &\le &
  e^{L\int_0^T\left(|Z^\e_s|+1\right)ds} 
  \left(\|W^\e\|_T+ \|K\|_T\right).
  \eeaa
Then,
 \beaa
 \eps^{-1\slash 2} e^{-LT} \dbE^{\dbP^{\e, Z^\e}}[\|\overline{\d}_T\|]
 &\le&
\dbE^{\dbP^{\e, Z^\e}}\Big[e^{L\int_0^T|Z^\e_s|ds}[\|W^\e\|_T+\|K\|_T\Big]\\
&\le& \Big( \dbE^{\dbP^{\e, Z^\e}}\Big[e^{2L\int_0^T|Z^\e_s|ds}\Big]\Big)^{1\over 2} \Big(\dbE^{\dbP^{\e, Z^\e}}\Big[\|W^\e\|_T^2+\|K\|_T^2\Big]\Big)^{1\over 2}.
 \eeaa
 Recall that $\si_t(0,x)$ is bounded. One may easily check that, for some constant  $C$ independent of $\e$,  
 \beaa
\dbE^{\dbP^{\e, Z^\e}}\Big[\|W^\e\|_T^2+\|K\|_T^2\Big] \le C.
 \eeaa
 Moreover, note that
 \beaa
 Y^\e_t = \xi + {1\over 2} \int_t^T |Z^\e_s|^2 ds - \sqrt{\e}\int_t^T  Z^\e_t d W^\e_t.
 \eeaa
Then, it follows that $ \|Z\|_{\Hbmo^2(\dbP^{\e, Z^\e})} \le 4 \|\xi\|_\infty$, and $\dbE^{\dbP^{\e,Z^\e}}\big[e^{\eta\int_0^T|Z^\e_s|^2ds} \big]  \le C$ for all $\e>0$,  for some $\eta>0$ and $C>0$ independent of $\e$, see e.g. \cite{CvitanicZhang}.  This implies $\dbE^{\dbP^{\e, Z^\e}}\Big[e^{2L\int_0^T|Z^\e_s|ds}\Big] \le C$ and thus
 \beaa
\dbE^{\dbP^{\e, Z^\e}}[\|\overline\d\|_T]\le C\sqrt{\e},\q\forall \e.
 \eeaa
 Similarly, $\dbE^{\dbP^{\e,Z^\e}}[\|\underline\delta\|_T] \le C\sqrt{\eps}$, and we may conclude that
 \beaa
 \dbE^{\dbP^{\e,Z^\e}}\Big[\rho\big(\big\|B-\o^{Z^\e}\big\|_T
                                                                +\big\|X-x^{Z^\e}\big\|_T
                                                         \big)
                                          \Big]
&\longrightarrow&
0,~~\mbox{as}~~\eps\searrow 0,
\eeaa
completing the proof.
\qed

\section{Asymptotics of the exiting probability}
\label{sect:Prob2}

This section is dedicated to the proof of Theorem \ref{thm:exit}. As before, we introduce the processes:
 \beaa
 Y^\e_t
 :=
 -\e\ln p^\e_t,
 ~~
 p^\e_t
 :=
 \dbP^\e_t[H<T]
 &\mbox{for all}&
 t\le T.
 \eeaa
Unlike the previous problem, the present example features an additional difficulty due to the singularity of the terminal condition:
 \beaa
 \lim_{t\rightarrow T}Y^\e_t
 =
 \infty
 &\mbox{on}&
 \{H\geq T\}.
 \eeaa
We shall first show that $\limsup_{\e\downarrow0}Y_0^\e\leq Q_0$. Adapting the argument of Fleming \& Soner \cite{FlemingSoner}, this will follow from the following estimate.
  
\begin{lem}\label{lem:upbound}
There exists a constant $K$ such that for any $\e>0$ we have
\beaa
Y^\e_t
\;\le\; \frac{Kd(X_t,\pa O)}{T-t}
&\text{for all}&
t<T~\text{ and }~t\leq H,~~\dbP^\e\text{-a.e.}
\eeaa
\end{lem}
\begin{proof}
First, fix $T_1<T$. For $x\in\dbR^d$, we denote by $x^1$ its first component. Since $O$ is bounded, there exists constant $m$ such that $x^1+\m>0$ for all $x\in O$. Define a function:
 \beaa
 g^\e(t,x)
 \;:=\;
 \exp\left(-\frac{\l(x^1+\m)}{\e(T_1-t)}\right),
 &\mbox{for}&
 t<T_1,~x\in \mbox{cl}(O),
 \eeaa
where $\l$ is some constant to be chosen later. By It\^o's formula, we have $\dbP^\e\text{-a.s.}$,
\beaa
dg^\e(t,X_t)=\frac{g^\e(t,X_t)}{\e (T_1-t)^2} \left[\frac12a^{1,1}_t(B,X)\l^2-\l(X^1_t+\m)-(T_1-t)\l b^1_t(B,X)\right]dt+dM_t,
\eeaa
for some $\dbP^\e-$martingale $M$. Since $a^{1,1}$ is uniformly bounded away from zero and $b^1$ is uniformly bounded, the $dt$-term of the above expression is positive for a sufficiently large $\l=\l^*$. Hence, $g^\e(t,X_t)$ is a submartingale on $[0,T_1\wedge H]$. Also, note that $g^\e(T_1,X_{T_1})=0\leq p_{T_1}^\e$ and $g^\e(H,X_H)\leq 1= p_H^\e$. Since $p^\e$ is a martingale, we conclude that
 \beaa
 g^\e(t,X_t)
 \;\le\; 
 p^\e_t
 &\mbox{for all}&
 t\leq T_1\wedge H,\ \dbP^\e\text{-a.s.}
 \eeaa
Denote $d(x):=d(x,\pa O)$. Since $\pa O$ is $C^3$, there exists a constant $\eta$ such that on $\{x\in O:d(x)<\eta\}$, the function $d$ is $C^2$.
Now, define
 \beaa
 \tilde{g}^\e(t,x)
 \;:=\;
 \exp\left(-\frac{Kd(x)}{\e(T_1-t)}\right),
 &\mbox{for}&
 t<T_1,~x\in \mbox{cl}(O),
 \eeaa 
for some $K\geq \frac{\l^*(C+\m)}{\eta}$. Clearly, for $t\leq T_1\wedge H$ and $d(X_t)\geq \eta$, we have
 $$
 \tilde{g}^\e(t,X_t)
 \;\le\; 
 g^\e(t,X_t)
 \;\le\; 
 p^\e_t,~~\dbP^\e-\text{a.s.}
 $$
In the remaining case $t\leq T_1\wedge H$ and $d(X_t)< \eta$, we will now verify that 
 \beaa
 \big\{\tilde{g}^\e(s,X_s)1_{\{d(X_t)<\eta\}},s\in[t,H_{\eta}\wedge H\wedge T]\big\}
 &\mbox{is a}&
 \dbP^\e-\mbox{submartingale,}  
 \eeaa 
where $H_{\eta}:=\inf\{s:d(X_s)\geq \eta\}$. By It\^o's formula, together with the fact that $|Dd(x)|=1$, 
 \beaa
 d\tilde{g}^\e(s,X_s) 
 &=& 
 \frac{K\tilde{g}^\e(s,X_s)}{\e(T_1-s)^2}
 \Big[\frac{K}{2}a_sDd(X_s)\cdot Dd(X_s)-\e\frac{T_1-s}{2}\text{tr}\big(a_sD^2d(X_s)\big)\\
 & & \hspace{25mm} 
 -(T_1-s)b_s\cdot Dd(X_s)-d(X_s)\Big]ds + dM_s
 \\ 
 &\ge& 
 \frac{K\tilde{g}^\e(s,X_s)}{\e(T_1-s)^2}\Big(\frac{K}{2}\d-\e\frac{T_1-s}{2}|a_s| \big|D^2d(X_s)\big|
                                                                     -(T_1-s)\|b_s\|\Big)ds
 +dM_s.
 \eeaa
Hence, for sufficiently large $K=K^*$, the $dt$-term is positive, and  $\tilde{g}^\e(s,X_s)1_{\{d(X_t)<\eta\}}$ is a submartingale for $s\in[t,H_{\eta}\wedge H\wedge T]$. We also verify directly that 
 $$
 \tilde{g}^\e(H_{\eta}\wedge H\wedge T,X_{H_{\eta}\wedge H\wedge T})1_{\{d(X_t)<\eta\}}
 \;\le\;
 p^\e_{H_{\eta}\wedge H\wedge T},
 ~~\dbP^\e-\text{a.s.}
 $$
Since $p^\e$ is a $\dbP^\eps-$martingale, we deduce that $\tilde{g}^\e(t,X_t)\leq p^\e_t$ for $t\leq T_1\wedge H$ and $d(X_t)<\eta$. Thus, we may conclude that 
$$\tilde{g}^\e(t,X_t)\leq p^\e_t\text{ for all }t\leq T_1\wedge H,\ \dbP^\e\text{-a.s.}$$
Let $T_1\rightarrow T$, we finally get
\beaa
Y^\e_t
\;\le\; 
\frac{Kd(X_t)}{T-t}
&\text{for all}&
t<T~~\text{and}~~t\leq H,~~\dbP^\e\text{-a.s.}
\eeaa
\end{proof}

\begin{prop}\label{prop:exit le}
$\limsup_{\e\downarrow 0}Y_0^\e\leq Q_0$.
\end{prop}

\begin{proof}
As in Proposition \ref{bsde}, we may show that there exists a process $Z^\e$ such that for any $T_1<T$:
 $$
 Y^\e_t
 =
 Y^\e_{T_1}-\frac12\int_t^{T_1}|Z^\e_s|^2ds+\int_t^{T_1}Z^\e_s\cdot dB_s,
 ~~\dbP^\e-\text{a.s.}
 $$
Define a sequence of BSDEs:
 $$
 \overline{Y}^{\e,T_1}_t
 =
 \frac{Kd(X_{T_1},O^c)}
        {T-T_1}
 -\frac12\int_t^{T_1}|Z_s^{\e,T_1}|^2ds
 +\int_t^{T_1} Z_t^{\e,T_1}\cdot dB_s,
 ~~\dbP^\e-\text{a.s.}
 $$
Note that $Y^\e_{T_1\wedge H}\leq \frac{Kd(X_{T_1\wedge H},O^c)}{T-T_1\wedge H}\leq \frac{Kd(X_{T_1},O^c)}{T-T_1}$. By Lemma \ref{lem:upbound} and the comparison principle of BSDE, we deduce that
 \beaa
 Y^\e_0
 &\le& 
 \overline{Y}^{\e,T_1}_0
 ~~\text{for all}~~T_1<T.
 \eeaa
Since $\xi(x):=\frac{Kd(x_{T_1},O^c)}{T-T_1}$ is bounded and uniformly continuous, it follows from Theorem \ref{thm:laplace} that
 \beaa
 \lim_{\e\rightarrow 0}\overline{Y}^{\e,T_1}_0
 =
 y^{T_1}_0
 :=
 \inf_{\a\in\dbL^2}\Big\{\frac12\int_0^{T_1}\a_t^2dt+\frac{Kd(x^\a_{T_1},O^c)}{T-T_1}\Big\}.
\eeaa
Thus, we have
\beaa
\limsup_{\e\downarrow 0}Y_0^\e 
&\leq &
 \inf_{\a\in\dbL^2}\Big\{\frac12\int_0^{T_1}\a_t^2dt+\frac{Kd(x^\a_{T_1},O^c)}{T-T_1}\Big\}
 \;\leq\;
 \inf_{\a\in\dbL^2,x^\a_{T_1}\notin O}
 \Big\{\frac12\int_0^T\a_t^2dt\Big\}.
 \eeaa
Finally, observe that 
 $$
 \inf_{\a\in\dbL^2,x^\a_{T_1}\notin O}
 \Big\{\frac12\int_0^T\a_t^2dt\Big\}
 =
 \inf_{\a\in\dbL^2,x^\a_{T_1\wedge\cdot}\notin \cO}
 \Big\{\frac12\int_0^T\a_t^2dt\Big\}
 \longrightarrow Q_0,
 ~~\text{as}~~T_1\rightarrow T.
 $$
\end{proof}

To complete the proof of Theorem \ref{thm:exit}, we next complement the result of Proposition \ref{prop:exit le} by the opposite inequality.

\begin{prop}\label{prop:exit ge}
$\liminf_{\e\downarrow 0}Y_0^\e\geq Q_0.$
\end{prop}

\begin{proof} 
We organize the proof in three steps.
\\
{\bf 1.} Define another sequence of BSDEs:
 \beaa
 \underline{Y}^{\e,T_1,m}_t
 =  
 md(X_{T_1},O^c)\wedge Y^\e_{T_1}-\frac12\int_t^{T_1}|Z_s^{\e,T_1,m}|^2ds +\int_t^{T_1} Z_t^{\e,T_1,m}\cdot dB_s,\ \dbP^\e\text{-a.s.}
\eeaa
By comparison of BSDEs, we have that $\underline{Y}^{\e,T_1,m}_t\leq Y^\e_t$ for all $t\leq T_1$. Then, by the stability of BSDEs, we know that $Y^{\e,T_1,m}$ converge to the solution of the following BSDE as $T_1\rightarrow T$:
 \beaa
 \underline{Y}^{\e,m}_t 
 &=&  
 md(X_{T},O^c)
 -\frac12\int_t^{T}|Z_s^{\e,m}|^2ds
 +\int_t^{T} Z_t^{\e,m}\cdot dB_s,
 ~~\dbP^\e\text{-a.s.}
 \eeaa
Again, we may apply Theorem \ref{thm:laplace} and get that
 \bea\label{ineq-propexitge}
 \liminf_{\eps\downarrow 0}Y^\eps_0
 \;\ge\;
 \lim_{\e\downarrow 0}\underline{Y}^{\e,m}_0
 &=&
 y^m_0:=\inf_{\a\in\dbL^2}\Big\{\frac12\int_0^T\a_s^2ds+md(x^\a_{T},O^c)\Big\}.
 \eea
{\bf 2.} We now prove that the sequence $\big(y_0^m\big)_m$ is bounded. Take $\a_t\equiv C\cdot 1$. Then
 $$
 x^\a_T
 =
 x_0+\int_0^T(b_t+C\si_t\cdot 1)dt.
 $$
Since $b$ is bounded and $\si$ is positive, when $C=C_0$ is sufficiently large, we will have $x^\a_T\notin O$. Hence, $y_0^m\leq \frac12 C_0^2Td.$
\\
\noindent {\bf 3.} In view of \eqref{ineq-propexitge}, we now conclude the proof of the proposition by verifying that $y_0^m\longrightarrow Q_0$, as $m\to\infty$. Let $\rho>0$. 
By the definition of $y_0^m$, there is a $\rho$-optimal $\a^{\rho}$:
 \beaa
 y_0^m+\rho 
 &>& 
 \frac12\int_0^T |\a_t^{\rho}|^2dt
 +md(x^{\rho}_{T},O^c),
 \eeaa
where we denoted $x^{\rho}:=x^{\a^{\rho}}$. By the boundedness of $(y^m_0)_m$ in Step 2, we have $d(x^{\rho}_T,O^c)\le\frac{C}{m}$. So, there exists a point $x_0\in \pa O$ such that $|x^{\rho}_T-x_0|\leq \frac{C}{m}$. 
Define:
 \beaa
 \tilde{\a}_t
 &:=&
 \a^{\rho}_t+\si^{-1}_t\frac{x_0-x^{\rho}_T}{T}.
 \eeaa
Then, $x^{\tilde{\a}}_T=x_0\notin O$. Also, note that $\si^{-1}_t\frac{x_0-x^{\rho}_T}{T}=o(\frac{1}{m})$ when $m\rightarrow\infty$. Hence,
\beaa
\frac12\int_0^T |\a_t^{\rho}|^2dt = \frac12\int_0^T |\tilde{\a}_t-\si^{-1}_t\frac{x_0-x^{\rho}_T}{T}|^2dt\geq \inf_{\a\in\dbL^2,x^\a_T\notin O}\Big\{\frac12\int_0^T |\a_t|^2dt\Big\}+o(\frac{1}{m}).
\eeaa 
Finally, sending $m\rightarrow \infty$, we see that $\lim_{m\rightarrow \infty}y_0^m +\rho\geq Q_0.$
Since $\rho$ is arbitrary, the proof is complete.
\end{proof}

\section{Viscosity property of the candidate solution}
\label{sect:existence}

This section is devoted to prove Theorem \eqref{thm: viscosity sol}.

\begin{lem}\label{estimat hat omega}
Fix $K\geq 0$. There exists a constant $C$ such that for any $t\in[0,T]$ and $\hat\o^1,\hat\o^2\in\hat\O$,
$$\sup_{\a:\int_t^T|\a|^2_sds\leq K}\|\hat\o^{\a,t,\hat\o^1}-\hat\o^{\a,t,\hat\o^2}\|\leq C\|\hat\o^1-\hat\o^2\|_t$$
\end{lem}
\begin{proof}
By the definition of $\hat\o^{\a,t,\hat\o^i}$ ($i=1,2$), we know that the components $\o^{\a,t,\hat\o^i}$ are equal. The difference comes from the component $x^{\a,t,\hat\o^i}$. Denote $\d x_t:=\|x^{\a,t,\hat\o^1}-x^{\a,t,\hat\o^2}\|^2_t$. Then, by the definition of $x^{\a,t,\hat\o^i}$ and the Lipschitz continuity of $b$ and $\si$, we obtain that
\beaa
\d x_s &\leq & \int_0^s C(\|\hat\o^1-\hat\o^2\|^2_t+\d x_r)dr+C\big(\int_0^s(\|\hat\o^1-\hat\o^2\|_t+\d x_r)|\a_r| dr\big)^2\\
& \leq & \int_0^s C(\|\hat\o^1-\hat\o^2\|^2_t+\d x_r)dr+2KC(\int_0^s (\|\hat\o^1-\hat\o^2\|^2_t+\d x_r)dr)
\eeaa
Finally, the claim results from the Gronwall's inequality.
\end{proof}

By standard argument, one may easily show the following dynamic programming for the optimal control problem \eqref{potential_sol}.

\begin{lem}[Dynamic programming]\label{lem:u_DPP}
Let $u$ be the value function defined in \eqref{potential_sol}. Then, for all $0\leq t\leq s\leq T$ and $\hat\o\in\hat\O$, we have
\beaa
u(t,\hat\o) 
 =
 \inf_{\a\in\dbL^2_d}
 \Big\{\frac12\int_t^s |\a_s|^2ds+u^{t,\hat\o}(s-t,\hat\o^{\a,t,\hat\o})\Big\},
 \eeaa
where $u^{t,\hat\o}(t',\hat\o'):=u(t+t',\hat\o\otimes_t \hat\o')$.
\end{lem}

\begin{lem}\label{lem: lipschitz}
The function $u$ defined in (\ref{potential_sol}) is bounded and Lipschitz-continuous.
\end{lem}

\begin{proof}
Clearly, $u$ inherits the bound of $\xi$. For $t\in [0,T]$, $\hat\o^1,\hat\o^2\in\hat\O$, since $\xi$ is bounded, there exists constant $K$ such that
\beaa
u_t(\hat\o^i)
& = & \inf_{\a\in\dbL^2_d}\Big\{\frac12\int_t^T|\a_s|^2ds+\xi^{t,\hat\o^i}(\hat\o^{\a,t,\hat\o^i})\Big\}\\
& = & \inf_{\a:\int_t^T|\a|^2_sds\leq K}\Big\{\frac12\int_t^T|\a_s|^2ds+\xi^{t,\hat\o^i}(\hat\o^{\a,t,\hat\o^i})\Big\}.
\eeaa
It follows from Lemma \ref{estimat hat omega} that:
 \be\label{estimate eta}
 \big|u(t,\hat\o^1)-u(t,\hat\o^2)\big|
 \le 
 \sup_{\a:\int_t^T|\a|^2_sds\leq K}
 \big\{\big|\xi^{t,\hat\o^1}(\hat\o^\a)-\xi^{t,\hat\o^2}(\hat\o^\a)
         \big|
 \big\}
 \le
 C\big\|\hat\o^1_{t\wedge\cdot}-\hat\o^2_{t\wedge\cdot}\big\|.
\ee
On the other hand, fixing $\hat\o$, it follows from the dynamic programming principle that
 \be\label{estimate time 1}
 u(t+h,\hat\o_{t\wedge\cdot})-u(t,\hat\o) 
 =
 \sup_{\a\in\dbL^2}
 \Big\{-\frac12\int_t^{t+h}\a^2_sds-u^{t,\hat\o}(h,\hat\o^{\a,t,\hat\o})+u(t+h,\hat\o_{t\wedge\cdot})\Big\}
 \ge
 0,
 \ee
where the last inequality is induced by the constant control $\alpha=0$.
Moreover, since $b$ and $\si$ are bounded, note that $\|(\hat\o\otimes_t\hat\o^{\a,t,\hat\o})_{(t+h)\wedge\cdot}-\hat\o_{t\wedge\cdot}\|\leq C\int_t^{t+h}(1+|\a_s|)ds$. Then, using again the dynamic programming principle together with (\ref{estimate eta}), we obtain
 \be\label{estimate time 2}
 u(t+h,\hat\o_{t\wedge\cdot})-u(t,\hat\o)
 \le 
 \sup_{\a\in\dbL^2}
 \Big\{\int_t^{t+h}\Big(-\frac12\a^2_s+C|\a_s|+C\Big)ds\Big\}
 \le
 \Big(\frac{C^2}{2}+C\Big)h.
 \ee
Combining this with (\ref{estimate eta}), we see that
 \beaa
 \big|u(t+h,\hat\o^1)-u(t,\hat\o^2)\big| 
 &\le& 
 \big|u(t+h,\hat\o^1)-u(t+h,\hat\o^1_{t\wedge\cdot})\big|
 \\
 &&
 +\big| u(t+h,\hat\o^1_{t\wedge\cdot})-u(t,\hat\o^1)\big|
 +\big|u(t,\hat\o^1)-u(t,\hat\o^2)\big|
 \\
 &\le& 
 C'(\|\hat\o^1\|_t^{t+h}+h+
     \|\hat\o^1_{t\wedge\cdot}-\hat\o^2_{t\wedge\cdot}\|)
 \\
 &\le& 
 3C'(h+\|\hat\o^1_{(t+h)\wedge
\cdot}-\hat\o^2_{t\wedge\cdot}\|).
\eeaa
\end{proof}

Now, consider a functional $u_K$:
\beaa
 u_K(t,\hat\o) := \inf_{\|\a\|_\infty\leq K} \Big[\xi(\hat\o \otimes_t \hat\o^{\a,t,\hat\o}) + {1\over 2} \int_t^T |\a_s|^2ds\Big];
\eeaa
Notice that $u_K\ge u_{K-1}\ge u$.

\begin{prop}
For $K$ sufficiently large, we have $u=u_K$. 
\end{prop}
\begin{proof}
Similar to Lemma \ref{lem: lipschitz}, for each $K$, one may easily see that $u_K(t,\cd)$ is uniformly Lipschitz in $\o$ with the same Lipschitz constant denoted as $L$. We first claim that there exists $\a^K$ such that
 \bea
 \label{oK}
 u_K(0,0) = \xi(\hat\o^{\a^K}) + {1\over 2} \int_0^T |\a^K_t|^2dt.
 \eea
 Then for any $t$ and $h$, one can easily show that 
 \beaa
  u_K(t,\hat\o^{\a^K}) = u_K(t+h, \hat\o^{\a^K}) + {1\over 2} \int_t^{t+h} |\a^K_s|^2ds.
 \eeaa
 On the other hand, by the dynamic programming,
 \beaa
  u_K(t,\hat\o^{\a^K}) \le u_K(t+h, \hat\o^{\a^K}_{t\wedge \cd}).
 \eeaa
 Then
 \begin{multline*}
 {1\over 2} \int_t^{t+h} |\a^K_s|^2ds \le u_K(t+h, \hat\o^{\a^K}_{t\wedge \cd}) - u_K(t+h, \hat\o^{\a^K})\\ 
 \le L \|\hat\o^{\a^K} - \hat\o^{\a^K}_{t\wedge \cd}\|_{t+h}\le CL\int_t^{t+h} (1+|\a^K_s|)ds,
 \end{multline*}
 where $C$ is a common bound for the coefficients $b$ and $\si$. Since $t$ and $h$ are arbitrary, we get $ \|\a^K\|_\infty\le C^{'}$ for some constant $C^{'}$ independent of $K$. Then $u_K = u_{C^{'}}$ for any $K\ge C^{'}$, and thus $u = u_{C^{'}}$.
 
 We now prove the existence claim \eqref{oK}. Let $\a^{K,n}$ be a minimum sequence of controls for $u_K(0,0)$, namely 
 \bea
\label{oKn}
 u_K(0,0) =\lim_{n\to\infty}\Big[ \xi(\hat\o^{\a^{K,n}}) + {1\over 2} \int_0^T |\a^{K,n}_t|^2dt\Big].
 \eea
By compactness of $\O_K$, the sequence $\{\o^{\a^{K,n}}, n\ge 1\}$ has a limit $\o^K\in \O_K$, after possibly passing to a subsequence:
 \bea
 \label{oKnconv}
 \lim_{n\to \infty}\|\o^{\a^{K,n}}-\o^K\|_T =0.
 \eea
  By \reff{oKn} and since $\xi$ is bounded,  it is clear that $\sup_n\int_0^T |\a^{K,n}_t|^2 dt<\infty$. Then without loss of generality we may assume $\{\a^{K,n}, n\ge 1\}$ converges to certain $\a^K$ weakly in $\dbL^2([0, T])$. Then for any $t$ and $h$,
  \beaa
  \o^K_{t+h} - \o^K_t = \lim_{n\to\infty} [  \o^{\a^{K,n}}_{t+h} - \o^{\a^{K,n}}_t] =  \lim_{n\to\infty} \int_t^{t+h}  \a^{K,n}_s ds =  \int_t^{t+h}  \a^K_s ds.
  \eeaa
  This implies that $\o^K = \o^{\a^K}$. Further, by Gronwall's inequality, we obtain that
\be\label{x conv} 
 \lim_{n\rightarrow\infty}\|x^{\a^{K,n}}-x^{\a^K}\|_T=0.
\ee
  
  Now by Mazur's lemma, there exist convex combinations  $\tilde \a^{K, n} = \sum_i c^n_i \a^{K, m^n_i}$, where $m^n_i \ge n$, such that $\{\tilde\a^{K,n}, n\ge 1\}$ converges to  $\a^K$ strongly in $\dbL^2([0, T])$. Then by Jensen's inequality we see that 
  \beaa
  \int_0^T |\a^K_t|^2 dt=  \lim_{n\to\infty} \int_0^T |\tilde\a^{K,n}_t|^2 dt \le \lim_{n\to\infty} \sum_i c^n_i \int_0^T | \a^{K,m^n_i}_t|^2 dt 
  \eeaa
  On the other hand, by \reff{oKnconv}, \reff{x conv}  and since $\xi$ is continuous, we have
  \beaa
  \xi(\hat\o^{\a^K}) =  \lim_{n\to\infty} \sum_i c^n_i \xi(\hat\o^{\a^{K, m^n_i}}).
  \eeaa
  Then
  \beaa
  \xi(\hat\o^{\a^{K}}) + {1\over 2} \int_0^T |\a^{K}_t|^2dt \le \lim_{n\to\infty} \sum_i c^n_i \Big[\xi(\hat\o^{\a^{K, m^n_i}}) + {1\over 2}\int_0^T |\a^{K,m^n_i}_t|^2 dt \Big] = u_K(0,0),
  \eeaa
  where the last equality follows from \reff{oKn}. This proves the claim.
\end{proof}

Clearly our equation \eqref{PPDE} satisfies the conditions of Lukoyanov \cite{Lukoyanov}, so that a comparison result for bounded viscosity super and subsolutions holds true. Conseuently, uniqueness holds for \eqref{PPDE} within the class of bounded functions and, in order to prove Theorem \ref{thm: viscosity sol} it remains to verify that $u$ satisfies the viscosity properties.

\vspace{5mm}

\no {\bf Proof of Theorem \ref{thm: viscosity sol}}\quad Fix $K_0$ such that $u=u_{K_0}$. Recall that $b$ and $\si$ are bounded by $C$. Then, define $K:=C(1+K_0)$, so that for all $\|\a\|_\infty \leq K_0$ and $\hat\o\in\hat\O_K$, we have $\hat\o^{\a,t,\hat\o}\in\hat\O_K$. 

We first prove the viscosity subsolution property. Let $(t,\hat\o)\in \Theta_{K}$, and $\f\in \underline{\cA}^K u(t,\hat\o)$. By the dynamic programming principle, we have:
 \bea\label{dpp}
 u(t,\hat\o)
 =
 \inf_{\a\in\dbL^2}\Big\{\frac12\int_t^{t+h}\a^2_rdr+u^{t,\hat\o}(h,\hat\o^{\a,t,\hat\o})\Big\}
 &\mbox{for}&
 h\geq 0.
 \eea
Since $\f\in\underline{\cA}^K u(t,\hat\o)$, we have for all $\|\a\|_\infty\leq K_0$:
 $$
 0
 \le
 \frac12 \int_t^{t+h} |\a|^2_rdr+u^{t,\hat\o}(h,\hat\o^{\a,t,\hat\o})-u(t,\hat\o)
 \le 
 \frac12 \int_t^{t+h} |\a|^2_rdr+\f^{t,\hat\o}(h,\hat\o^{\a,t,\hat\o})-\f(t,\hat\o).
 $$
By the smoothness of $\f$, this provides:
 \be\label{for a}
 0
 \le
 \frac{1}{h}\int_0^{h} \big(\pa_t\f+b\pa_x\f+\frac12|\a|^2
                                          +\a\cdot(\pa_\o\f+\si^{\rm T}\pa_x\f)
                                   \big)^{t,\hat\o}(r,\hat\o^{\a,t,\hat\o})dr.
 \ee
By sending $h\rightarrow 0$, we obtain
 \beaa
 -\left(\pa_t\f+b_{\cdot}\pa_x\f+\inf_{|\a|\leq K_0}\Big(\frac12|\a|^2
                                          +\a\cdot(\pa_\o\f+\si^{\rm T}\pa_x\f)\Big)\right)(t,\hat\o)
 \le
 0.
 \eeaa

We next prove the viscosity supersubsolution property. Assume not, then there exists  $\f \in \overline{\cA}^{K} u(t,\hat\o)$ such that
\beaa
c :=  -\left(\pa_t\f+b_{\cdot}\pa_x\f+\inf_{|\a|\leq K_0}\Big(\frac12|\a|^2
                                          +\a\cdot(\pa_\o\f+\si^{\rm T}\pa_x\f)\Big)\right)(t,\hat\o) >0.
\eeaa
Without loss of generality, we may assume that $\f(t,\hat\o)=u(t,\hat\o)$. Recall that $u=u_{K_0}$. Now for any $h>0$, by the dynamic programming,
\beaa
\f(t,\hat\o) = u(t,\hat\o) 
& = &
 \inf_{\|\a\|_\infty\leq K_0} \Big[u^{t,\hat\o}_h(\hat\o^{\a,t\hat\o}) + {1\over 2}\int_t^{t+h} |\a_s|^2 ds\Big]\\
 & \ge &
  \inf_{\|\a\|_\infty\leq K_0} \Big[\f^{t,\hat\o}_h(\hat\o^{\a,t,\hat\o}) + {1\over 2}\int_t^{t+h} |\a_s|^2 ds\Big].
\eeaa
Then, 
\beaa
0 &\ge& \inf_{\|\a\|_\infty\leq K_0} \Big[\f^{t,\hat\o}_h(\hat\o^{\a,t,\hat\o})-\f_t(\hat\o) + {1\over 2}\int_t^{t+h} |\a_s|^2 ds\Big]\\
&=&  \inf_{\|\a\|_\infty\leq K_0}  \int_0^h \Big[\pa_t\f+b_{\cdot}\pa_x\f+ \frac12|\a|^2 +\a\cdot(\pa_\o\f+\si^{\rm T}\pa_x\f)\Big]^{t,\hat\o}(s,\hat\o^{\a,t,\hat\o}) ds\\
&\ge&  \inf_{\|\a\|_\infty\leq K_0}  \int_0^h \Big[c-C\Big(|\pa_t \f^{t,\hat\o}(s,\hat\o^{\a,t,\hat\o})- \pa_t \f(t,\hat\o)|+|\pa_{\hat\o}\f^{t,\hat\o}(s,\hat\o^{\a,t,\hat\o})-\pa_{\hat\o}\f(t,\hat\o)|\Big)\Big]ds \\
&\ge&  \Big[c-\rho\big(d_\infty((1+K)h\big) \Big] h,
\eeaa
which leads to a contradiction by choosing $h$ sufficiently small.
\qed

\end{document}